\documentclass[12pt,reqno]{amsart}
\usepackage{amsfonts}

\usepackage{amssymb,amsmath,graphicx,amsfonts,euscript}
\usepackage{color}

\newtheorem{thm}{Theorem}[section]

\newtheorem{Pros}{Proposition}[section]
\newtheorem{lemma}{Lemma}[section]

\theoremstyle{definition}
\newtheorem{define}{Definition}[section]

\theoremstyle{remark}
\newtheorem{rem}{Remark}[section]

\allowdisplaybreaks \voffset=-0.2in \numberwithin{equation}{section}

\begin{document}
\bigskip

\centerline{\Large\bf  Global regularity for the 2D micropolar  }
\smallskip

\centerline{\Large\bf  Rayleigh-B$\rm\acute{e}$nard  convection system with velocity  }

\centerline{\Large\bf  zero dissipation and temperature critical }

\centerline{\Large\bf dissipation }

\bigskip

\centerline{ Baoquan Yuan\footnote{Corresponding Author: B. Yuan} \ and\ Changhao Li}

\medskip
\centerline{School of Mathematics and Information Science, Henan Polytechnic University,
}
\medskip

\centerline{  Henan,  454000,  China.}

\medskip

\centerline{(E-mail: \texttt{bqyuan@hpu.edu.cn, lch1162528727@163.com})}

\medskip
\bigskip

{\bf Abstract:}~~%
This paper studies the global regularity problem for the 2D micropolar
Rayleigh-B$\rm\acute{e}$nard convection system with velocity zero dissipation, micro-rotation velocity Laplace dissipation and temperature critical dissipation. By introducing a combined quantity and using the technique of Littlewood-Paley decomposition, we establish the global regularity result of solutions to this system.

{\vskip 1mm


}

\vskip .3in
\section{Introduction}
In this paper, we consider a mathematical model depicting in the framework of the Boussinesq approximation, the heat convection in the micropolar fluid \cite{KLL2019,KL2020,T2006}, which is called micropolar Rayleigh--B$\rm\acute{e}$nard convection system. The micropolar Rayleigh--B\'{e}nard convection system is a coupled system, including the micropolar equation and Rayleigh--B\'{e}nard equation. The micropolar equations were first introduced in 1966 by Eringen \cite{E1966}, which enables us to consider some physical phenomena that cannot be treated by the classical Navier-Stokes equations for the viscous incompressible fluids, for example, the motion of animal blood, liquid crystals and dilute aqueous polymer solutions, etc. The Rayleigh-B$\rm\acute{e}$nard equation can be used to model the behavior of the fluid layer filling the region between two rigid surfaces heated from below. The standard micropolar Rayleigh--B$\rm\acute{e}$nard convection system is as follows
\begin{equation}\label{1.1}
\left\{\aligned
&\partial_{t}u+u\cdot\nabla u-(\mu+\chi)\Delta u+\nabla p=2\chi\nabla\times\omega+e_{3}\theta , \\
&\partial_{t}\omega+u\cdot\nabla\omega-\nu\Delta\omega+4\chi\omega-\eta\nabla\nabla\cdot\omega=2\chi\nabla\times u, \\
&\partial _{t}\theta+u\cdot\nabla\theta-\kappa\Delta\theta=u\cdot e_{3}, \\
&\nabla\cdot u=0,\\
&u(x,0)=u_{0}(x),\omega(x,0)=\omega_{0}(x),\theta(x,0)=\theta_{0}(x),
\endaligned\right.
\end{equation}
where $u$ denotes the velocity of the fluid, $\omega$ denotes the micro-rotation velocity, $\theta$ is the scalar temperature, $p$ is the pressure. $\mu$, $\chi$ and $\kappa$ are the kinematic viscosity, the vortex viscosity and the thermal diffusivity, respectively. $\eta$ and $\nu$ are angular viscosities. The forcing term $\theta e_{3}$ in the momentum equation describes the action of the buoyancy force on fluid motion and $u\cdot e_{3}$ models the Rayleigh-B\'{e}nard convection in a heated inviscid fluid.

If $\chi=0$, $\omega=0$ and the Rayleigh--B$\rm\acute{e}$nard convection term $u\cdot e_{3}=0$, system (\ref{1.1}) becomes the Boussinesq equation,  which models geophysical flows such as atmospheric fronts and oceanic circulation, and plays an important role in the study of Rayleigh-B\'{e}nard convection. Due to its physical application and mathematical significance, Boussinesq equations have attracted considerable attention in recent years and have made some progress \cite{C2006,AHM2020,HK2007,HKR2010,Y2016}. When the effect of temperature is neglected, namely $\theta=0$, then system ({\ref{1.1}}) reduces to the classical micropolar system. More advances about the micropolar equations with partial dissipation refer to \cite{DLW2017,DZ2010,GR1977}.

If $\chi=0$ and $\omega=0$, system ({\ref{1.1}}) reduces  to the B$\rm\acute{e}$nard equation.
In 2012, Wu and Xue \cite{WX2012} established Global well-posedness for the 2D inviscid B\'enard system with temperature fractional diffusivity, which is shown below:
\begin{equation}\label{1.2}
\left\{\aligned
&\partial_{t}u+u\cdot\nabla u+\nabla p=e_{2}\theta, \\
&\partial_{t}\theta+u\cdot\nabla\theta+\Lambda^{\beta}\theta=u\cdot e_{2}, \\
&\nabla\cdot u=0, \\
&(u,\theta)(x,t)_{t=0}=(u_{0},\theta_{0})(x),~~~~x\in\mathbb{R}^{2},
\endaligned\right.
\end{equation}
where $1<\beta<2$ and $\Lambda=(-\Delta)^{\frac{1}{2}}$ denotes the Zygmund operator. Wu and Xue \cite{WX2012} obtained the Key estimate of $\|\theta\|_{L^{\infty}}$ by the DeGiorgi-Nash estimate approach. In 2017, Ye \cite{Y2017} provided a simple approach to obtain the $\|\theta\|_{L^{\infty}}$ estimate of system (\ref{1.2}).
When $\beta=1$, the system (\ref{1.2}) corresponds to the 2D inviscid B\'enard system with temperature critical diffusivity, and it is difficult to show the global regularity in this case by the methods of  Wu and Xue  \cite{WX2012} and  Ye \cite{Y2017}. The main obstacle is the lack of the $L^{\infty}$-information of $\theta$.
In this paper, we investigate the case of $\beta=1$, which is called the critical case, coupled with the micropolar equation.
The specific equation is as follows
\begin{equation}\label{1.3}
\left\{\aligned
&\partial_{t}u+u\cdot\nabla u+\nabla p=2\chi\nabla\times\omega+e_{2}\theta, \\
&\partial_{t}\omega+u\cdot\nabla\omega-\nu\Delta\omega+4\chi\omega=2\chi\nabla\times u, \\
&\partial_{t}\theta+u\cdot\nabla\theta+\Lambda\theta=u\cdot e_{2}, \\
&\nabla\cdot u=0, \\
&(u,\omega,\theta)(x,t)_{t=0}=(u_{0},\omega_{0},\theta_{0})(x),~~~~x\in\mathbb{R}^{2}.
\endaligned\right.
\end{equation}
The standard 2D micropolar Rayleigh--B$\rm\acute{e}$nard convection can be written as
\begin{equation}\label{1.4}
\left\{\aligned
&\partial_{t}u+u\cdot\nabla u-(\mu+\chi)\Delta u+\nabla p=2\chi\nabla\times\omega+e_{2}\theta, \\
&\partial_{t}\omega+u\cdot\nabla\omega-\nu\Delta\omega+4\chi\omega=2\chi\nabla\times u, \\
&\partial_{t}\theta+u\cdot\nabla\theta-\kappa\Delta\theta=u\cdot e_{2}, \\
&\nabla\cdot u=0, \\
&(u,\omega,\theta)(x,t)_{t=0}=(u_{0},\omega_{0},\theta_{0})(x),~~~~x\in\mathbb{R}^{2}.
\endaligned\right.
\end{equation}
In 2020, Xu and Chi \cite{XC2020} obtained global regularity for the system (\ref{1.4}) with temperature zero diffusivity (i.e., $\kappa=0$). In 2021, Wang \cite{W2021} established the global regularity of the system (\ref{1.4}) without velocity dissipation (i.e., $\mu+\chi=0$); Deng and Shang \cite{DS2021} obtained global regularity for the system (\ref{1.4}) with only velocity dissipation (i.e., $\nu=\kappa=0$). In this paper, we study the global regularity of the system (\ref{1.3}), which is equivalent to the system (\ref{1.4}) with velocity zero dissipation and temperature critical dissipation (i.e., $\mu+\chi=0$ and $-\Delta\theta$ is replaced by $\Lambda\theta$). The main result is stated as follows.

\begin{thm}\label{thm1.1}
Consider ({\ref{1.3}}) with $\chi>0$ and $\nu>0$, Assume $(u_{0},\omega_{0},\theta_{0})\in H^{s}(\mathbb{R}^{2})$ with $s>2$ and $\nabla\cdot u_{0}=0$. Then the system ({\ref{1.3}}) has a unique global solution $(u,\omega,\theta)$ such that for any $T>0$,
\begin{align*}
u&\in C([0,T);H^{s}(\mathbb{R}^{2})),\\
\omega\in C([0,T);&H^{s}(\mathbb{R}^{2}))\cap L^{2}(0,T;H^{s+1}(\mathbb{R}^{2})),\\
\theta\in C([0,T);&H^{s}(\mathbb{R}^{2}))\cap L^{2}(0,T;H^{s+\frac{1}{2}}(\mathbb{R}^{2})).
\end{align*}
\end{thm}
\begin{rem}\label{Remark1.1}
The results of Wang \cite{W2021} is improved by Theorem \ref{thm1.1}, in which we reduce the temperature dissipation term $-\Delta\theta$ to $\Lambda\theta$.
\end{rem}

\begin{rem}\label{Remark1.2}
If $\chi=0$, $\omega=0$ the  2D micropolar Rayleigh--B$\rm\acute{e}$nard convection system (\ref{1.3}) reduces to the 2D Euler-Boussinesq-B$\rm\acute{e}$nard equations. With Yudovich's type data, Ye \cite{Ye2023} also obtain the global unique solution using a different method.

\end{rem}

To prove the Theorem \ref{thm1.1}, a large portion of the efforts are devoted to obtaining the a priori estimates for $(u,\omega,\theta)$. In this paper, we shall consider the vorticity $\Omega\triangleq\nabla\times u$, which satisfies
\begin{align}\label{1.5}
\partial_{t}\Omega+u\cdot\nabla\Omega=-\Delta\omega+\partial_{1}\theta.
\end{align}

To deal with the terms on the right hand side of the vorticity equation, we construct a combined quantity $\Gamma\triangleq\Omega+\mathcal{R}_{1}\theta+\omega$ with ${R}_{1}\triangleq\partial_{1}\Lambda^{-1}$. We estimate $\|\theta\|_{L^{\infty}}$ by doing the $\dot{H}^{k}$ estimate of $\theta$ and the $L^{r}$ estimates of $\Gamma$ and $\omega$, where $\frac{1}{2}<k\leq\frac{r-2}{r}$ and $4< r<\infty$.
Then, the desired estimate $\|\Omega\|_{L^{\infty}}$ can be obtained from the estimates of $\|\Gamma\|_{L^{\infty}}$, $\|\mathcal{R}_{1}\theta\|_{L^{\infty}}$ and $\|\omega\|_{L^{\infty}}$. Furthermore, with the aid of the property of the heat kernel (see Lemma \ref{lem2.9}), we can obtain the estimate of $\|\nabla\omega\|_{L^{\infty}}$. We further do the $\dot{H}^{1}$ estimate of the temperature equation and derive the estimate of $\int_{0}^{T}\|\Lambda^{\frac{3}{2}}\theta\|_{L^{2}}^{2}\mathrm{d}t$. With the disposal of the previous estimates at hand, we are able to show the global $H^{s}$ estimates of $(u,\omega,\theta)$ by using the logarithmic Sobolev and Gronwall's inequalities (see Lemmas \ref{lem2.6}, \ref{lem2.7}).

Throughout this paper, the letter $C$ denotes a generic constant whose exact value may change from line to line, although the same letter $C$ is used. Because the exact values of $\mu$, $\chi$ and $\nu$ do not affect the result, for simplicity, we set  $\mu=\chi=\frac{1}{2}$ and $\nu=1$.

\vskip .3in
\section{Preliminaries}
In this section we first introduce Littlewood-Paley decomposition and the definition of Besov spaces. Materials presented here can be found in several books and papers \cite{BCD2011,M2004,T1983}. Given $f(x)\in\mathcal{S}(\mathbb{R}^{d})$, the Schwartz class of rapidly decreasing functions, define the Fourier transform as
\begin{align*}
\hat{f}(\xi)=\mathcal{F}f(\xi)=(2\pi)^{-d/2}\int_{\mathbb{R}^{d}}e^{-ix\cdot\xi}f(x)\mathrm{d}x,
\end{align*}
and its inverse Fourier transform:
\begin{align*}
\check{f}(\xi)=\mathcal{F}^{-1}f(\xi)=(2\pi)^{-d/2}\int_{\mathbb{R}^{d}}e^{ix\cdot\xi}f(\xi)\mathrm{d}\xi.
\end{align*}
Choose a nonnegative radial function $\chi\in C_{0}^{\infty}(\mathbb{R}^{d})$ such that $0\leq\chi(\xi)\leq1$ and
\begin{align*}
\chi(\xi)=\begin{cases}
{ \begin{array}{ll} 1,\ \ \ \ \mathrm{for}\ |\xi|\leq\frac{3}{4},\\
0,\ \ \ \ \mathrm{for}\ |\xi|>\frac{4}{3},
 \end{array} }
\end{cases}
\end{align*}
and let $\hat{\varphi}(\xi)=\chi(\xi/2)-\chi(\xi)$, $\chi_{j}(\xi)=\chi(\frac{\xi}{2^{j}})$ and $\hat{\varphi}_{j}(\xi)=\hat{\varphi}(\frac{\xi}{2^{j}})$ for $j\in\mathbb{Z}$. Write
\begin{align*}
h(x)~~&=~~\mathcal{F}^{-1}\chi(x),\\
h_{j}(x)~~&=~~2^{dj}h(2^{j}x),\\
\varphi_{j}(x)~~&=~~2^{dj}\varphi(2^{j}x).
\end{align*}
Define the inhomogeneous operators $S_{j}$ and $\triangle_{j}$, respectively, as
\begin{align*}
\triangle_{-1}u(x)~~&=~~h\ast u(x),\\
\triangle_{j}u(x)~~&=~~\varphi_{j}\ast u(x)~~\mathrm{for}~j\geq0,\\
\triangle_{j}u(x)~~&=~~0~~\mathrm{for}~j\leq-2,\\
S_{j}u(x)~~&=\sum_{-1\leq k\leq j-1}\Delta_{k}u(x)~~\mathrm{for}~j\geq0.
\end{align*}
Formally $\triangle_{j}$ is a frequency projection to the annulus $|\xi|\approx2^{j}$, while $S_{j}$ is a frequency projection to the ball $|\xi|\lesssim2^{j}$ for $j\geq0$. For any $u(x)\in L^{2}(\mathbb{R}^{d})$ we have the inhomogeneous Littlewood-Paley decomposition
\begin{align*}
u(x)~~=~~h\ast u(x)+\sum_{j\geq0}\varphi_{j}\ast u(x).
\end{align*}
Clearly,
\begin{align*}
\mathrm{supp}\chi(\xi)\cap\mathrm{ supp}\hat{\varphi}_{j}(\xi)~~&=~~\emptyset~\mathrm{for}~j\geq1,\\
\mathrm{supp}\hat{\varphi}_{j}(\xi)\cap \mathrm{supp}\hat{\varphi}_{j'}(\xi)~~&=~~\emptyset~\mathrm{for}~|j-j'|\geq2.
\end{align*}

We now recall the definition of the Besov space.
\begin{define}\label{define2.1}
Let $s\in\mathbb{R}$ and $1\leq p,~q\leq+\infty$, the Besov space $B_{p,q}^{s}(\mathbb{R}^{d})$ abbreviated as $B_{p,q}^{s}$ is defined by
\begin{align*}
B_{p,q}^{s}=\{f(x)\in \mathcal{S}'(\mathbb{R}^{d});\|f\|_{B_{p,q}^{s}}<+\infty\},
\end{align*}
where
\begin{align*}
\|f\|_{B_{p,q}^{s}}=\begin{cases}
{ \begin{array}{ll} (\sum_{j\geq-1}2^{jsq}\|\Delta_{j} f\|_{L^{p}}^{q})^{\frac{1}{q}},~~~\mathrm{for}~q<+\infty,\\
\sup_{j\geq-1}2^{js}\|\Delta_{j} f\|_{L^{p}},~~~~~~~\mathrm{for}~q=+\infty
 \end{array} }
\end{cases}
\end{align*}
is the Besov norm.
\end{define}

Next we recall the definition of two kinds of space-time Besov spaces.
\begin{define}\label{defin2.2}
Let $T>0$, $s\in\mathbb{R}$, $1\leq p,~q,~r\leq \infty$ and $j\geq-1$.\\
(1) We call $u\in L^{r}(0,T;B_{p,q}^{s}(\mathbb{R}^{d}))$ if and only if
\begin{align*}
\|u\|_{L_{T}^{r}B_{p,q}^{s}}\triangleq\|(2^{js}\|\Delta_{j}u\|_{L^{p}})_{l^{q}}\|_{L_{T}^{r}}<\infty.\\
\end{align*}
(2) We call $u\in \tilde{L}^{r}(0,T;B_{p,q}^{s}(\mathbb{R}^{d}))$ if and only if
\begin{align*}
\|u\|_{\tilde{L}_{T}^{r}B_{p,q}^{s}}\triangleq\|2^{js}\|\Delta_{j}u\|_{L_{T}^{r}L^{p}}\|_{l^{q}}<\infty.\\
\end{align*}
\end{define}

The following lemma states the  Bernstein's inequalities (see \cite{BCD2011,CMZ2007}).
\begin{lemma}\label{lem2.1}
(1) Let $\alpha>0$ and $1\leq p\leq q\leq\infty$, if f satisfies
\begin{align*}
\mathrm{supp}\hat{f}\subset\{\xi\in\mathbb{R}^{d}:~|\xi|\leq K2^{j}\}
\end{align*}
for some integer $j$ and a constant $K>0$, then
\begin{align*}
\|(-\triangle)^{\alpha}f\|_{L^{q}}\leq C_{1}2^{2\alpha j+jd(\frac{1}{p}-\frac{1}{q})}\|f\|_{L^{p}}.
\end{align*}
(2) Let $\alpha>0$ and $1\leq p\leq q\leq\infty$, if f satisfies
\begin{align*}
\mathrm{supp}\hat{f}\subset\{\xi\in\mathbb{R}^{d}:~K_{1}2^{j}\leq|\xi|\leq K_{2}2^{j}\}
\end{align*}
for some integer $j$ and constants $0<K_{1}\leq K_{2}$, then
\begin{align*}
C_{1}2^{2\alpha j}\|f\|_{L^{q}}\leq\|(-\triangle)^{\alpha}f\|_{L^{q}}\leq C_{2}2^{2\alpha j+jd(\frac{1}{p}-\frac{1}{q})}\|f\|_{L^{p}},
\end{align*}
where $C_{1}$ and $C_{2}$ are two constants depending on $\alpha$, $p$ and $q$.\\
(3) Let $0\leq\alpha\leq 2$ and $2\leq p<\infty$. Then there exist two positive constants $c_{p}$ and $C_{p}$ such that for any $f\in S'$ and $j\geq0$, we have
\begin{align*}
c_{p}2^{\frac{2\alpha j}{p}}\|\Delta_{j}f\|_{L^{p}}\leq
\|\Lambda^{\alpha}(|\Delta_{j}f|^{\frac{p}{2}})\|_{L^{2}}^{\frac{2}{p}}\leq
C_{p}2^{\frac{2\alpha j}{p}}\|\Delta_{j}f\|_{L^{p}}.
\end{align*}
\end{lemma}

Next we recall some commutator estimates.
\begin{lemma}\label{lem2.2}
\cite{KPV1991,KP1988} Let  $d$ be the space dimension.
Let $s>0$, $1<r<\infty$ and $\frac{1}{r}=\frac{1}{p_{1}}+\frac{1}{q_{1}}=\frac{1}{p_{2}}+\frac{1}{q_{2}}$ with $q_{1},p_{2}\in(1,\infty)$ and $p_{1},q_{2}\in[1,\infty]$. Then
\begin{align*}
\|[\Lambda^{s},f]g\|_{L^{r}}\leq C(\|\nabla f\|_{L^{p_{1}}}\|\Lambda^{s-1}g\|_{L^{q_{1}}}+\|\Lambda^{s}f\|_{L^{p_{2}}}\|g\|_{L^{q_{2}}}),
\end{align*}
where $[\Lambda^{s},f]g=\Lambda^{s}(fg)-f\Lambda^{s}g$ and $C$ is a constant depending on the indices $s,r,p_{1},q_{1},p_{2},q_{2}$.
\end{lemma}

\begin{lemma}\label{lem2.3} \cite{Y2017}
Let  $d$ be the space dimension and $\frac{1}{p}=\frac{1}{p_{1}}+\frac{1}{p_{2}}$ with $p\in[2,\infty)$,
$p_{1}$, $p_{2}$ $\in[2,\infty]$, $q\in[1,\infty]$ as well as $s\in(-1,1-\epsilon)$ for $\epsilon\in(0,1)$. Then
\begin{align*}
\|[\Lambda^{\epsilon},f\cdot\nabla]g\|_{B_{p,q}^{s}}\leq C(\|\nabla f\|_{L^{p_{1}}}\|g\|_{B_{p_{2},q}^{s+\epsilon}}+\|f\|_{L^{2}}\|g\|_{L^{2}}),
\end{align*}
where f is a divergence-free vector field.
\end{lemma}

\begin{lemma} \label{lem2.4} \cite{Ye2023,HKR2011} Let  $d$ be the space dimension. Let $u$ be a smooth divergence-free vector field and $\theta$ be a smooth scalar function. Then\\
(1) For every $p\in(1,\infty)$ there exists a constant $C=C(p)$ such that
\begin{align*}
\|[\mathcal{R}_{1},u\cdot\nabla]\theta\|_{L^{p}}\leq C\|\nabla u\|_{L^{p}}\|\theta\|_{L^{\infty}}.
\end{align*}\\
(2) For every $(p,r)\in(1,\infty)\times[1,\infty]$ and $\epsilon>0$
there exists a constant $C=C(p,r,\epsilon)$ such that
\begin{align*}
\|[\mathcal{R}_{1},u\cdot\nabla]\theta\|_{B_{\infty,r}^{0}}\leq C(\|\Omega\|_{L^{\infty}}+\|\Omega\|_{L^{p}})(\|\theta\|_{B_{\infty,r}^{\epsilon}}+\|\theta\|_{L^{p}}).
\end{align*}
\end{lemma}

\begin{lemma} \label{lem2.5} \cite{HKR2011} Let  $d$ be the space dimension. Let $u$ be a smooth divergence-free vector field. Then, for all $p\in[1,\infty]$ and $q\geq-1$,
\begin{align*}
\|[\Delta_{q},u\cdot\nabla]\theta\|_{L^{p}}\leq C\|\nabla u\|_{L^{p}}\|\theta\|_{B_{\infty,\infty}^{0}},
\end{align*}
for every smooth scalar function $\theta$.
\end{lemma}

Now, we recall the logarithmic Sobolev and Gronwall's inequalities (see \cite{Y2018}).
\begin{lemma}\label{lem2.6}
Let  $d$ be the space dimension and $s>\frac{d}{2}+1$. Then there exists two constants $C_{1}$ and $C_{2}$ such that
\begin{align*}
\|\nabla f\|_{L^{\infty}}\leq C_{1}\|\nabla f\|_{L^{2}}+C_{2}\|\nabla f\|_{B^{0}_{\infty,\infty}}\mathrm{log}(e+\|\Lambda^{s}f\|_{L^{2}}).
\end{align*}
\end{lemma}

\begin{lemma}\label{lem2.7}
 Let $A$ and $B$ be two absolutely continuous and nonnegative functions on $(0,T)$ for given $T>0$ satisfying
\begin{align*}
A'(t)+B(t)\leq[m(t)+n(t)\mathrm{log}(A+B+e)](A+e)+f(t).
\end{align*}
for any $t\in(0,T)$, where $m(t)$, $n(t)$ and $f(t)$ are all nonnegative and integrable functions on $(0,T)$. Assume further that there are three constants $K\in[0,\infty)$, $\alpha\in[0,\infty)$ and $\beta\in[0,1)$, such that for any $t\in(0,T)$,
\begin{align*}
n(t)\leq K(A(t)+e)^{\alpha}(A(t)+B(t)+e)^{\beta}.
\end{align*}
Then the following estimate holds true
\begin{align}\label{2.1}
A(t)+\int^{t}_{0}B(s)ds\leq\tilde{C}(m,n,f,\alpha,\beta,K,t)<\infty,
\end{align}
for any $t\in(0,T)$. Especially, for the case $\alpha=0$, namely
\begin{align*}
n(t)\leq K(A(t)+B(t)+e)^{\beta},
\end{align*}
({\ref{2.1}}) still holds true.
\end{lemma}

The following form of positive lemma gives a specific lower bound, which makes it convenient for us to do the $L^{r}$ estimate for $r>2$.
\begin{lemma}\label{lem2.8}\cite{N2005}
Suppose $s\in[0,2]$, $f$ and $\Lambda^{s}f\in L^{p}$ for $p\geq2$. Then
\begin{align*}
\int_{\mathbb{R}^{2}}|f|^{p-2}f\Lambda^{s}f\mathrm{d}x\geq\frac{2}{p}\int_{\mathbb{R}^{2}}(\Lambda^{\frac{s}{2}}|f|^{\frac{p}{2}})^{2}\mathrm{d}x.
\end{align*}
\end{lemma}

Finally, we give a classic $L^{p}$-$L^{q}$ type estimate for the heat operator as follows.
\begin{lemma}\label{lem2.9}\cite{SS2005}
Let  $d$ be the space dimension, $s>0$ and $1\leq p\leq q\leq\infty$. For any $t>0$, we have
\begin{align*}
\|\nabla^{s}e^{-\Delta t}f\|_{L^{q}}\leq Ct^{-\frac{s}{2}}t^{-\frac{d}{2}(\frac{1}{p}-\frac{1}{q})}\|f\|_{L^{p}}.
\end{align*}
\end{lemma}

\vskip .3in
\section{The proof of Theorem \ref{thm1.1}}
\setcounter{equation}{0}
\vskip .1in
\subsection{The estimate of $\|\theta\|_{L^{\infty}}$}\

In this subsection, we first establish the global $L^{2}$-bound for $(u,\omega,\theta)$, and then the estimate of $\|\theta\|_{L^{\infty}}$.
Precisely, we have the following lemma.
\begin{lemma}\label{lem3.1}
Let $(u_{0},\omega_{0},\theta_{0})$ satisfy the assumptions stated in Theorem {\ref{thm1.1}}. Then the corresponding solution $(u,\omega,\theta)$ of ({\ref{1.3}}) obeys the following bounds, for any $0<t<T$
\begin{align}\label{3.1}
\|u\|^{2}_{L^{2}}+\|\omega\|^{2}_{L^{2}}+\|\theta\|^{2}_{L^{2}}+\int^{t}_{0}\|\nabla\omega(\cdot,\tau)\|^{2}_{L^{2}}\mathrm{d}\tau
+\int^{t}_{0}\|\Lambda^{\frac{1}{2}}\theta(\cdot,\tau)\|^{2}_{L^{2}}\mathrm{d}\tau
\leq\|(u_{0},\omega_{0},\theta_{0})\|_{L^{2}}^{2}e^{CT},
\end{align}
\begin{align}\label{3.2}
\|\theta\|_{L^{\infty}}\leq Ce^{CT}.
\end{align}
\end{lemma}

\begin{proof} First, we dot ({\ref{1.3}}) with $(u,\omega,\theta)$. By the H\"{o}lder and Young's inequalities, one has
\begin{align*}
&\frac{1}{2}\frac{\mathrm{d}}{\mathrm{d}t}(\|u\|_{L^{2}}^{2}+\|\omega\|_{L^{2}}^{2}+\|\theta\|_{L^{2}}^{2})
+\|\nabla\omega\|_{L^{2}}^{2}+2\|\omega\|_{L^{2}}^{2}+\|\Lambda^{\frac{1}{2}}\theta\|_{L^{2}}^{2}\\
&=\int_{\mathbb{R}^{2}}(\nabla\times\omega)\cdot u\mathrm{d}x+\int_{\mathbb{R}^{2}}e_{2}\theta\cdot u\mathrm{d}x
+\int_{\mathbb{R}^{2}}(\nabla\times u)\cdot\omega\mathrm{d}x+\int_{\mathbb{R}^{2}}u\cdot e_{2}\theta\mathrm{d}x\\
&\leq C\|u\|_{L^{2}}\|\nabla\omega\|_{L^{2}}+C\|u\|_{L^{2}}\|\theta\|_{L^{2}}\\
&\leq C\|u\|_{L^{2}}^{2}+C\|\theta\|_{L^{2}}^{2}+\frac{1}{2}\|\nabla\omega\|_{L^{2}}^{2}.
\end{align*}
By Gronwall's inequality, ({\ref{3.1}}) is proved.

Now, we estimate $\|\theta\|_{L^{p}}$ for $2<p<\infty$. Doing the $L^{p}$ estimate of $\theta$, we have
\begin{align}\label{3.3}
\frac{1}{p}\frac{\mathrm{d}}{\mathrm{d}t}\|\theta\|_{L^{p}}^{p}+
\int_{\mathbb{R}^{2}}\Lambda\theta\cdot(|\theta|^{p-2}\theta)\mathrm{d}x
=\int_{\mathbb{R}^{2}}u_{2}\cdot(|\theta|^{p-2}\theta)\mathrm{d}x.
\end{align}
By Lemma \ref{lem2.8}, one has
\begin{align}\label{3.4}
\int_{\mathbb{R}^{2}}\Lambda\theta\cdot(|\theta|^{p-2}\theta)\mathrm{d}x\geq
\frac{2}{p}\int_{\mathbb{R}^{2}}(\Lambda^{\frac{1}{2}}\theta^{\frac{p}{2}})^{2}\mathrm{d}x=
\frac{2}{p}\|\Lambda^{\frac{1}{2}}|\theta|^{\frac{p}{2}}\|_{L^{2}}^{2}.
\end{align}
By the Interpolation and Young's inequalities, we have
\begin{align}\label{3.5}
\nonumber |\int_{\mathbb{R}^{2}}u_{2}\cdot(|\theta|^{p-2}\theta)\mathrm{d}x|\leq&C\|u\|_{L^{2}}
\||\theta|^{\frac{p}{2}}\|_{L^{4(p-1)/p}}^{\frac{2(p-1)}{p}}\\
\nonumber \leq&C\|u\|_{L^{2}}\||\theta|^{\frac{p}{2}}\|_{L^{2}}^{\frac{2}{p}}
\|\Lambda^{\frac{1}{2}}|\theta|^{\frac{p}{2}}\|_{L^{2}}^{\frac{2(p-2)}{p}}\\
\leq&C\|u\|_{L^{2}}\|\theta\|_{L^{p}}
\|\Lambda^{\frac{1}{2}}|\theta|^{\frac{p}{2}}\|_{L^{2}}^{\frac{2(p-2)}{p}}\\
\nonumber \leq&\frac{1}{p}\|\Lambda^{\frac{1}{2}}|\theta|^{\frac{p}{2}}\|_{L^{2}}^{2}+C\|\theta\|_{L^{p}}^{p}
+C\frac{1}{p}(Cp)^{p}.
\end{align}
Inserting (\ref{3.4}) and (\ref{3.5}) into (\ref{3.3}), one has
\begin{align*}
\frac{\mathrm{d}}{\mathrm{d}t}\|\theta\|_{L^{p}}^{p}+\|\Lambda^{\frac{1}{2}}|\theta|^{\frac{p}{2}}\|_{L^{2}}^{2}
\leq C\|\theta\|_{L^{p}}^{p}+C(Cp)^{p}.
\end{align*}
By Gronwall's inequality, we have
\begin{align}\label{3.6}
\|\theta\|_{L^{p}}^{p}+\int_{0}^{t}\|\Lambda^{\frac{1}{2}}|\theta|^{\frac{p}{2}}\|_{L^{2}}^{2}\mathrm{d}
\tau\leq C(\|\theta_{0}\|_{L^{p}}^{p}+(Cp)^{p})e^{CT}.
\end{align}

Next, we estimate $\|\Lambda^{k}\theta\|_{L^{2}}$, $\|\Gamma\|_{L^{r}}$ and $\|\omega\|_{L^{r}}$ for $\frac{1}{2}<k\leq\frac{r-2}{r}$ and $4<r<\infty$.
Applying the operator $\Lambda^{k}$ to the temperature equation, and dotting $\Lambda^{k}\theta$ on the both sides and integrating over $\mathbb{R}^{2}$, we can derive that
\begin{align}\label{3.7}
\frac{1}{2}\frac{\mathrm{d}}{\mathrm{d}t}\|\Lambda^{k}\theta\|_{L^{2}}^{2}+\|\Lambda^{k+\frac{1}{2}}\theta\|_{L^{2}}^{2}
=-\int_{\mathbb{R}^{2}}\Lambda^{k}(u\cdot\nabla\theta)\Lambda^{k}\theta\mathrm{d}x+
\int_{\mathbb{R}^{2}}\Lambda^{k}u_{2}\Lambda^{k}\theta\mathrm{d}x.
\end{align}
By Lemmas \ref{lem2.1} and \ref{lem2.3}, one has
\begin{align}\label{3.8}
\nonumber |\int_{\mathbb{R}^{2}}\Lambda^{k}(u\cdot\nabla\theta)\Lambda^{k}\theta\mathrm{d}x|
=&|\int_{\mathbb{R}^{2}}[\Lambda^{k},u\cdot\nabla]\theta\Lambda^{k}\theta\mathrm{d}x|\\
\nonumber \leq&C\|[\Lambda^{k},u\cdot\nabla]\theta\|_{H^{-\frac{1}{2}}}\|\Lambda^{k}\theta\|_{H^{\frac{1}{2}}}\\
\nonumber \leq&C\|[\Lambda^{k},u\cdot\nabla]\theta\|_{B_{2,2}^{-\frac{1}{2}}}(\|\Lambda^{k+\frac{1}{2}}\theta\|_{L^{2}}
+\|\Lambda^{k}\theta\|_{L^{2}})\\
\nonumber
\leq&C(\|\nabla u\|_{L^{r}}\|\theta\|_{B_{\frac{2r}{r-2},2}^{k-\frac{1}{2}}}+
\|u\|_{L^{2}}\|\theta\|_{L^{2}})(\|\Lambda^{k+\frac{1}{2}}\theta\|_{L^{2}}+\|\Lambda^{k}\theta\|_{L^{2}})\\
\nonumber \leq&C(\|\Omega\|_{L^{r}}\|\theta\|_{H^{{\frac{1}{2}}}}+
\|u\|_{L^{2}}\|\theta\|_{L^{2}})(\|\Lambda^{k+\frac{1}{2}}\theta\|_{L^{2}}+\|\Lambda^{k}\theta\|_{L^{2}})\\
\leq&C\|\Lambda^{\frac{1}{2}}\theta\|_{L^{2}}^{2}\|\Omega\|_{L^{r}}^{2}+C+
\frac{1}{2}\|\Lambda^{k+\frac{1}{2}}\theta\|_{L^{2}}^{2}+C\|\Lambda^{k}\theta\|_{L^{2}}^{2},
\end{align}
where we used $k\leq\frac{r-2}{r}$.
\begin{align}\label{3.9}
\nonumber |\int_{\mathbb{R}^{2}}\Lambda^{k}u_{2}\Lambda^{k}\theta\mathrm{d}x|
\leq&C\|\Lambda^{k}u_{2}\|_{L^{2}}\|\Lambda^{k}\theta\|_{L^{2}}\\
\nonumber \leq&C\|u\|_{L^{2}}^{1-\frac{kr}{2r-2}}\|\Omega\|_{L^{r}}^{\frac{kr}{2r-2}}\|\Lambda^{k}\theta\|_{L^{2}}\\
\leq&C(\|\Lambda^{k}\theta\|_{L^{2}}^{2}+\|\Omega\|_{L^{r}}^{2})+C.
\end{align}
Inserting (\ref{3.8}) and (\ref{3.9}) into (\ref{3.7}), one has
\begin{align}\label{3.10}
\nonumber \frac{\mathrm{d}}{\mathrm{d}t}\|\Lambda^{k}\theta\|_{L^{2}}^{2}+\|\Lambda^{k+\frac{1}{2}}\theta\|_{L^{2}}^{2}
\leq&C(\|\Lambda^{\frac{1}{2}}\theta\|_{L^{2}}^{2}+1)(\|\Lambda^{k}\theta\|_{L^{2}}^{2}+\|\Omega\|_{L^{r}}^{2})+C\\
\leq&C(\|\Lambda^{\frac{1}{2}}\theta\|_{L^{2}}^{2}+1)
(\|\Lambda^{k}\theta\|_{L^{2}}^{2}+\|\Gamma\|_{L^{r}}^{2}+\|\omega\|_{L^{r}}^{2}+1).
\end{align}
To close the estimate, we need to estimate the $L^{r}$ norms of $\Gamma$ and $\omega$.
Applying the operator $\mathcal{R}_{1}$ to the temperature equation, we have
\begin{align*}
\partial_{t}\mathcal{R}_{1}\theta+u\cdot\nabla\mathcal{R}_{1}\theta+\Lambda\mathcal{R}_{1}\theta=
-[\mathcal{R}_{1},u\cdot\nabla]\theta+\mathcal{R}_{1}u_{2}.
\end{align*}
Combining above equation, vorticity equation ({\ref{1.5}}) with the equation of $\omega$, we can get the equation of $\Gamma:=\Omega+\mathcal{R}_{1}\theta+\omega$ as
\begin{align*}
\partial_{t}\Gamma+u\cdot\nabla\Gamma=-[\mathcal{R}_{1},u\cdot\nabla]\theta+\mathcal{R}_{1}u_{2}+\Omega-2\omega,
\end{align*}
where $[\mathcal{R}_{1},u\cdot\nabla]\theta\triangleq\mathcal{R}_{1}(u\cdot\nabla\theta)-u\cdot\nabla\mathcal{R}_{1}\theta$ denotes the commutator.
Doing the $L^{r}$ estimate of $\Gamma$, we have
\begin{align}\label{3.11}
\frac{\mathrm{d}}{\mathrm{d}t}\|\Gamma\|_{L^{r}}^{2}\leq\|[\mathcal{R}_{1},u\cdot\nabla]\theta\|_{L^{r}}\|\Gamma\|_{L^{r}}
+(\|\mathcal{R}_{1}u_{2}\|_{L^{r}}+\|\Omega\|_{L^{r}}+2\|\omega\|_{L^{r}})\|\Gamma\|_{L^{r}}\triangleq I_{1}+I_{2}.
\end{align}
By Lemma \ref{lem2.4}, we have
\begin{align}\label{3.12}
\nonumber |I_{1}|\leq&C\|[\mathcal{R}_{1},u\cdot\nabla]\theta\|_{L^{r}}\|\Gamma\|_{L^{r}}\\
\nonumber \leq&C\|\nabla u\|_{L^{r}}\|\theta\|_{L^{\infty}}\|\Gamma\|_{L^{r}}\\
\leq&C(\|\Gamma\|_{L^{r}}+\|\mathcal{R}_{1}\theta\|_{L^{r}}+\|\omega\|_{L^{r}})
\|\theta\|_{L^{\infty}}\|\Gamma\|_{L^{r}}\\
\nonumber \leq& C\|\theta\|_{L^{\infty}}(\|\Gamma\|_{L^{r}}^{2}+\|\omega\|_{L^{r}}^{2}+1).
\end{align}
It is easy to derive that
\begin{align}\label{3.13}
|I_{2}|\leq C(\|\Gamma\|_{L^{r}}^{2}+\|\omega\|_{L^{r}}^{2}+1).
\end{align}
Inserting (\ref{3.12}) and (\ref{3.13}) into (\ref{3.11}), one has
\begin{align}\label{3.14}
\frac{\mathrm{d}}{\mathrm{d}t}\|\Gamma\|_{L^{r}}^{2}\leq C\|\theta\|_{L^{\infty}}(\|\Gamma\|_{L^{r}}^{2}+\|\omega\|_{L^{r}}^{2}+1).
\end{align}
By doing the $L^{r}$ estimate of $\omega$, we have
\begin{align}\label{3.15}
\frac{\mathrm{d}}{\mathrm{d}t}\|\omega\|_{L^{r}}^{2}+\|\omega\|_{L^{r}}^{2}\leq C\|\Omega\|_{L^{r}}\|\omega\|_{L^r}\leq C(\|\Gamma\|_{L^{r}}^{2}+\|\omega\|_{L^{r}}^{2}+1).
\end{align}
Adding (\ref{3.10}), (\ref{3.14}) and (\ref{3.15}) together and we get
\begin{align}\label{3.16}
\nonumber &\frac{\mathrm{d}}{\mathrm{d}t}(\|\Lambda^{k}\theta\|_{L^{2}}^{2}+\|\Gamma\|_{L^{r}}^{2}+\|\omega\|_{L^{r}}^{2})
+\|\Lambda^{k+\frac{1}{2}}\theta\|_{L^{2}}^{2}\\
\leq &C(\|\theta\|_{L^{\infty}}+\|\Lambda^{\frac{1}{2}}\theta\|_{L^{2}}^{2}+1)
(\|\Lambda^{k}\theta\|_{L^{2}}^{2}+\|\Gamma\|_{L^{r}}^{2}+\|\omega\|_{L^{r}}^{2}+1).
\end{align}
By using Lemma \ref{lem2.1} and the equation (\ref{3.6}), for $N$ to be determined later we have that for $k>\frac{1}{2}$
\begin{align*}
\|\theta\|_{L^{\infty}}\leq&\|S_{N+1}\theta\|_{L^{\infty}}+\sum_{q\geq N+1}\|\Delta_{q}\theta\|_{L^{\infty}}\\
\leq&C2^{N\frac{2}{r'}}\|S_{N+1}\theta\|_{L^{r'}}+C\sum_{q\geq N+1}2^{q(\frac{1}{2}-k)}\|\Lambda^{k+\frac{1}{2}}\Delta_{q}\theta\|_{L^{2}}\\
\leq &C2^{\frac{2N}{r'}}(r'+1)+C2^{N(\frac{1}{2}-k)}\|\Lambda^{k+\frac{1}{2}}\theta\|_{L^{2}}.
\end{align*}
 We choose $2N\le r'<\infty$ and $N$ satisfying
\begin{align*}
N\ge \Big[C\mathrm{log}(e+\|\Lambda^{k+\frac{1}{2}}\theta\|_{L^{2}}^{2})\Big]+1,
\end{align*}
thus we arrive at
\begin{align*}
\|\theta\|_{L^{\infty}}\leq C+C\mathrm{log}(e+\|\Lambda^{k+\frac{1}{2}}\theta\|_{L^{2}}^{2}).
\end{align*}
Inserting the above estimate for $\|\theta\|_{L^\infty}$ into (\ref{3.16}), we have
\begin{align*}
&\frac{\mathrm{d}}{\mathrm{d}t}(\|\Lambda^{k}\theta\|_{L^{2}}^{2}+\|\Gamma\|_{L^{r}}^{2}+\|\omega\|_{L^{r}}^{2})
+\|\Lambda^{k+\frac{1}{2}}\theta\|_{L^{2}}^{2}\\
\leq &C(\mathrm{log}(e+\|\Lambda^{k+\frac{1}{2}}\theta\|_{L^{2}}^{2})+\|\Lambda^{\frac{1}{2}}\theta\|_{L^{2}}^{2}+1)
(\|\Lambda^{k}\theta\|_{L^{2}}^{2}+\|\Gamma\|_{L^{r}}^{2}+\|\omega\|_{L^{r}}^{2}+1).
\end{align*}
Applying the logarithmic Gronwall's inequality in Lemma \ref{lem2.7}, we have that for $\frac{1}{2}<k\leq\frac{r-2}{r}$ and $4< r<\infty$
\begin{align*}
\|\Lambda^{k}\theta\|_{L^{2}}^{2}+\|\Gamma\|_{L^{r}}^{2}+
\|\omega\|_{L^{r}}^{2}+\int_{0}^{t}\|\Lambda^{k+\frac{1}{2}}\theta\|_{L^{2}}^{2}\mathrm{d}\tau\leq Ce^{CT}.
\end{align*}
By the Interpolation inequality, one has
\begin{align*}
\int_{0}^{t}\|u\|_{L^{\infty}}\mathrm{d}\tau
\leq&C\int_{0}^{t}\|u\|_{L^{2}}^{\frac{r-2}{2r-2}}\|\Omega\|_{L^{r}}^{\frac{r}{2r-2}}\mathrm{d}\tau\\
\leq&C\int_{0}^{t}(\|\Gamma\|_{L^{r}}+\|\omega\|_{L^{r}}+1)\mathrm{d}\tau\\
\leq&Ce^{CT}.
\end{align*}
We thus get from the temperature equation
\begin{align*}
\|\theta\|_{L^{\infty}}\leq C\|\theta_{0}\|_{L^{\infty}}+\int_{0}^{t}\|u\|_{L^{\infty}}\mathrm{d}\tau\leq Ce^{CT},
\end{align*}
and the proof of Lemma \ref{lem3.1} is complete.

\end{proof}

\subsection{The estimate of $\|\Omega\|_{L^{\infty}}$}\

In this subsection, we establish the estimate of $\|\Omega\|_{L^{\infty}}$.
Precisely, we have the following lemma.
\begin{lemma}\label{lem3.2}
Let $(u_{0},\omega_{0},\theta_{0})$ satisfy the assumptions stated in Theorem {\ref{thm1.1}}. Then the corresponding solution $(u,\omega,\theta)$ of ({\ref{1.3}}) obeys the following bounds
\begin{align*}
\|\Omega\|_{L^{\infty}}\leq Ce^{e^{CT}}
\end{align*}
for any $0<t<T$.
\end{lemma}
\begin{proof}
First, we estimate $\|\theta\|_{\tilde{L}_{t}^{1}B_{p,\infty}^{1}}$. Applying the operator $\Delta_{q}$ to the temperature equation, we have
\begin{align*}
\partial_{t}\Delta_{q}\theta+u\cdot\nabla\Delta_{q}\theta+\Lambda\Delta_{q}\theta
=-[\Delta_{q},u\cdot\nabla]\theta+\Delta_{q}u_{2}.
\end{align*}
Taking the $L^{2}$-inner product to the above equation with $|\Delta_{q}\theta|^{p-2}\Delta_{q}\theta$ for $2<p<\infty$ and $q\geq0$, we have
\begin{align*}
\frac{1}{p}&\frac{\mathrm{d}}{\mathrm{d}t}\|\Delta_{q}\theta\|_{L^{p}}^{p}+
\int_{\mathbb{R}^{2}}\Lambda\Delta_{q}\theta|\Delta_{q}\theta|^{p-2}\Delta_{q}\theta\mathrm{d}x\\
\leq&\|[\Delta_{q},u\cdot\nabla]\theta\|_{L^{p}}\|\Delta_{q}\theta\|_{L^{p}}^{p-1}+
\|\Delta_{q}u_{2}\|_{L^{p}}\Delta_{q}\theta\|_{L^{p}}^{p-1}.
\end{align*}
By Lemmas \ref{lem2.1} and \ref{lem2.8}, one has
\begin{align*}
\frac{c}{p}2^{q}\|\Delta_{q}\theta\|_{L^{p}}^{p}\leq
\int_{\mathbb{R}^{2}}\Lambda\Delta_{q}\theta|\Delta_{q}\theta|^{p-2}\Delta_{q}\theta\mathrm{d}x.
\end{align*}
Thus
\begin{align*}
\frac{\mathrm{d}}{\mathrm{d}t}\|\Delta_{q}\theta\|_{L^{p}}+c2^{q}\|\Delta_{q}\theta\|_{L^{p}}
\leq C\|[\Delta_{q},u\cdot\nabla]\theta\|_{L^{p}}+C\|\Delta_{q}u_{2}\|_{L^{p}}.
\end{align*}
Integrating with respect to $t$, it arrives at
\begin{align}\label{theta-p}
\|\Delta_{q}\theta\|_{L^{p}}\leq Ce^{-ct2^{q}}\|\Delta_{q}\theta_{0}\|_{L^{p}}+
C\int_{0}^{t}e^{-c(t-\tau)2^{q}}(\|[\Delta_{q},u\cdot\nabla]\theta\|_{L^{p}}+\|\Delta_{q}u_{2}\|_{L^{p}})
\mathrm{d}\tau.
\end{align}
Multiplying the above estimate (\ref{theta-p}) with $2^{q}$ and integrating in time and using Lemma \ref{lem2.5}, we have
\begin{align*}
2^q\|\Delta_{q}\theta\|_{L_{t}^{1}L^{p}}\leq&C\|\Delta_{q}\theta_{0}\|_{L^{p}}+
C\int_{0}^{t}(\|[\Delta_{q},u\cdot\nabla]\theta\|_{L^{p}}+\|\Delta_{q}u_{2}\|_{L^{p}})\mathrm{d}\tau\\
\leq&C\|\theta_{0}\|_{L^{p}}+C\int_{0}^{t}(\|\nabla u\|_{L^{p}}\|\theta\|_{B_{\infty,\infty}^{0}}+\|u\|_{L^{p}})\mathrm{d}\tau\\
\leq&C\|\theta_{0}\|_{L^{p}}+C\int_{0}^{t}(\|\Omega\|_{L^{p}}\|\theta\|_{L^{\infty}}+\|u\|_{L^{p}})\mathrm{d}\tau\\
\leq&Ce^{CT}.
\end{align*}
Thus we have
\begin{align*}
\|\theta\|_{\tilde{L}_{t}^{1}B_{p,\infty}^{1}}\leq Ce^{CT}.
\end{align*}
By Bernstein inequality in Lemma \ref{lem2.1}, we have that for every $0<\epsilon<1-\frac{2}{p}$ and $2<p<\infty$
\begin{align*}
\|\theta\|_{L_{t}^{1}B_{\infty,1}^{\epsilon}}\leq C\|\theta\|_{\tilde{L}_{t}^{1}B_{p,\infty}^{1}}\leq Ce^{CT}.
\end{align*}

Next, we estimate $\|\Omega\|_{L^{\infty}}$. By doing the $L^{\infty}$ estimates of $\Gamma$, $\omega$ and $\mathcal{R}_{1}\theta$, one has
\begin{align*}
\|\Gamma\|_{L^{\infty}}&\leq\|\Gamma_{0}\|_{L^{\infty}}+\int_{0}^{t}\|[\mathcal{R}_{1},u\cdot\nabla]\theta\|_{L^{\infty}}\mathrm{d}\tau
+\int_{0}^{t}\|\mathcal{R}_{1}u_{2}\|_{L^{\infty}}\mathrm{d}\tau\\
&+\int_{0}^{t}\|\Omega\|_{L^{\infty}}\mathrm{d}\tau+2\int_{0}^{t}\|\omega\|_{L^{\infty}}\mathrm{d}\tau,
\end{align*}
\begin{align*}
\|\omega\|_{L^{\infty}}+2\int_{0}^{t}\|\omega\|_{L^{\infty}}\mathrm{d}\tau\leq \|\omega_{0}\|_{L^{\infty}}+\int_{0}^{t}\|\Omega\|_{L^{\infty}}\mathrm{d}\tau,
\end{align*}
\begin{align*}
\|\mathcal{R}_{1}\theta\|_{L^{\infty}}\leq \|\mathcal{R}_{1}\theta_{0}\|_{L^{\infty}}
+\int_{0}^{t}\|[\mathcal{R}_{1},u\cdot\nabla]\theta\|_{L^{\infty}}\mathrm{d}\tau
+\int_{0}^{t}\|\mathcal{R}_{1}u_{2}\|_{L^{\infty}}\mathrm{d}\tau.
\end{align*}
Thus
\begin{align}\label{3.17}
\nonumber \|\Omega\|_{L^{\infty}}&\leq\|\Gamma\|_{L^{\infty}}+\|\mathcal{R}_{1}\theta\|_{L^{\infty}}+\|\omega\|_{L^{\infty}}\\
&\leq C\|\Omega_{0}\|_{L^{\infty}}+C\|\mathcal{R}_{1}\theta_{0}\|_{L^{\infty}}+C\|\omega_{0}\|_{L^{\infty}}+
C\int_{0}^{t}\|[\mathcal{R}_{1},u\cdot\nabla]\theta\|_{L^{\infty}}\mathrm{d}\tau\\
\nonumber &+C\int_{0}^{t}\|\mathcal{R}_{1}u_{2}\|_{L^{\infty}}\mathrm{d}\tau
+C\int_{0}^{t}\|\Omega\|_{L^{\infty}}\mathrm{d}\tau.
\end{align}
By Lemma \ref{lem2.4}, it can be derived
\begin{align}\label{3.18}
\nonumber \int_{0}^{t}\|[\mathcal{R}_{1},u\cdot\nabla]\theta\|_{L^{\infty}}\mathrm{d}\tau\leq&
C\int_{0}^{t}\|[\mathcal{R}_{1},u\cdot\nabla]\theta\|_{B_{\infty,1}^{0}}\mathrm{d}\tau\\
\nonumber \leq&C\int_{0}^{t}(\|\Omega\|_{L^{\infty}}+\|\Omega\|_{L^{2}})(\|\theta\|_{B_{\infty,1}^{\epsilon}}+\|\theta\|_{L^{2}})\mathrm{d}\tau\\
\leq&C\int_{0}^{t}(\|\theta\|_{B_{\infty,1}^{\epsilon}}+1)(\|\Omega\|_{L^{\infty}}+1)\mathrm{d}\tau.
\end{align}
For $2<p<\infty$ one has
\begin{align}\label{3.19}
\int_{0}^{t}\|\mathcal{R}_{1}u_{2}\|_{L^{\infty}}\mathrm{d}\tau\leq C\int_{0}^{t}\|u\|_{L^{2}}^{\frac{p-2}{2p-2}}\|\Omega\|_{L^{p}}^{\frac{p}{2p-2}}\mathrm{d}\tau\leq Ce^{CT}.
\end{align}
Inserting (\ref{3.18}) and (\ref{3.19}) into (\ref{3.17}), we have
\begin{align*}
\|\Omega\|_{L^{\infty}}\leq C+C\int_{0}^{t}(\|\theta\|_{B_{\infty,1}^{\epsilon}}+1)(\|\Omega\|_{L^{\infty}}+1)\mathrm{d}\tau.
\end{align*}
By Gronwall's inequality, we have
\begin{align*}
\|\Omega\|_{L^{\infty}}\leq Ce^{e^{CT}},
\end{align*}
and the proof of Lemma \ref{lem3.2} is complete.
\end{proof}

\subsection{The estimates of $\|\nabla\omega\|_{L^{\infty}}$ and $\int_{0}^{T}\|\Lambda^{\frac{3}{2}}\theta\|_{L^{2}}^{2}\mathrm{d}t$}\

In this subsection, we establish the estimates of $\|\nabla\omega\|_{L^{\infty}}$ and $\int_{0}^{T}\|\Lambda^{\frac{3}{2}}\theta\|_{L^{2}}^{2}\mathrm{d}t$.
Precisely, we have the following lemma.
\begin{lemma}\label{lem3.3}
Let $(u_{0},\omega_{0},\theta_{0})$ satisfy the assumptions stated in Theorem {\ref{thm1.1}}. Then the corresponding solution $(u,\omega,\theta)$ of ({\ref{1.3}}) obeys the following bounds
\begin{align*}
\|\nabla\omega\|_{L^{\infty}}+\|\Lambda\theta\|_{L^{2}}^{2}+\int_{0}^{t}\|\Lambda^{\frac{3}{2}}\theta\|_{L^{2}}^{2}\mathrm{d}\tau
\leq Ce^{e^{CT}}
\end{align*}
for any $0<t<T$.
\end{lemma}
\begin{proof}
First, we write the equation of $\omega$ as
\begin{align*}
\omega=e^{t\Delta}\omega_{0}+\int_{0}^{t}e^{(t-\tau)\Delta}(\Omega-2\omega-u\cdot\nabla\omega)\mathrm{d}\tau.
\end{align*}
We do the $\dot{W}^{1,\infty}$ estimate of above equation. By Lemma \ref{lem2.9} and Young's inequality, we have
\begin{align*}
\|\nabla\omega\|_{L^{\infty}}\leq&\|\nabla e^{t\Delta}\omega_{0}\|_{L^{\infty}}+
\int_{0}^{t}\|\nabla e^{(t-\tau)\Delta}(\Omega-2\omega-u\cdot\nabla\omega)\|_{L^{\infty}}\mathrm{d}\tau\\
\leq&\|\nabla\omega_{0}\|_{L^{\infty}}+\int_{0}^{t}(t-\tau)^{-\frac{1}{2}}(\|\Omega\|_{L^{\infty}}+2\|\omega\|_{L^{\infty}}
+\|u\|_{L^{\infty}}\|\nabla\omega\|_{L^{\infty}})\mathrm{d}\tau\\
\leq&\|\nabla\omega_{0}\|_{L^{\infty}}+
CT^{\frac{1}{4}}(\int_{0}^{t}(e^{e^{CT}}+e^{CT}\|\nabla\omega\|_{L^{\infty}})^{4}\mathrm{d}\tau)^{\frac{1}{4}},
\end{align*}
By the Gronwall's inequality, we have
\begin{align*}
\|\nabla\omega\|_{L^{\infty}}\leq Ce^{e^{CT}}.
\end{align*}

Next, we do the $\dot{H}^{1}$ estimate of $\theta$. Applying the operator $\Lambda$ to the temperature equation, dotting $\Lambda\theta$ and integrating over $\mathbb{R}^{2}$, we can derive that
\begin{align}\label{3.20}
\frac{1}{2}\frac{\mathrm{d}}{\mathrm{d}t}\|\Lambda\theta\|_{L^{2}}^{2}+\|\Lambda^{\frac{3}{2}}\theta\|_{L^{2}}^{2}
=-\int_{\mathbb{R}^{2}}\Lambda(u\cdot\nabla\theta)\Lambda\theta\mathrm{d}x+
\int_{\mathbb{R}^{2}}\Lambda u_{2}\Lambda\theta\mathrm{d}x.
\end{align}
By Lemma \ref{lem2.2} and Young's inequality, we have
\begin{align}\label{3.21}
\nonumber |\int_{\mathbb{R}^{2}}\Lambda(u\cdot\nabla\theta)\Lambda\theta\mathrm{d}x|
=&|\int_{\mathbb{R}^{2}}[\Lambda,u\cdot\nabla]\theta\Lambda\theta\mathrm{d}x|\\
\nonumber \leq&C\|[\Lambda,u\cdot\nabla]\theta\|_{L^{\frac{4}{3}}}\|\Lambda\theta\|_{L^{4}}\\
\nonumber \leq&C\|\nabla u\|_{L^{4}}\|\Lambda\theta\|_{L^{2}}\|\Lambda^{\frac{3}{2}}\theta\|_{L^{2}}\\
\nonumber \leq&C\|\Omega\|_{L^{4}}\|\Lambda\theta\|_{L^{2}}\|\Lambda^{\frac{3}{2}}\theta\|_{L^{2}}\\
\leq&C\|\Lambda\theta\|_{L^{2}}^{2}\|\Omega\|^2_{L^{4}}+\frac{1}{2}\|\Lambda^{\frac{3}{2}}\theta\|_{L^{2}}^{2}.
\end{align}
\begin{align}\label{3.22}
\nonumber |\int_{\mathbb{R}^{2}}\Lambda u_{2}\Lambda\theta\mathrm{d}x|
\leq&C\|\Lambda u\|_{L^{2}}\|\Lambda\theta\|_{L^{2}}\\
\leq&C\|\Lambda\theta\|_{L^{2}}^{2}+C.
\end{align}
Inserting (\ref{3.21}) and (\ref{3.22}) into (\ref{3.20}), one has
\begin{align*}
\frac{\mathrm{d}}{\mathrm{d}t}\|\Lambda\theta\|_{L^{2}}^{2}+
\|\Lambda^{\frac{3}{2}}\theta\|_{L^{2}}^{2}\leq C\|\Lambda\theta\|_{L^{2}}^{2}\|\Omega\|^2_{L^{4}}+C,
\end{align*}
thus we have
\begin{align*}
\|\Lambda\theta\|_{L^{2}}^{2}+
\int_{0}^{t}\|\Lambda^{\frac{3}{2}}\theta\|_{L^{2}}^{2}\mathrm{d}\tau\leq Ce^{e^{CT}},
\end{align*}
and the proof of Lemma \ref{lem3.3} is completed.
\end{proof}

\subsection{The global $H^{s}$ estimates of $(u,\omega,\theta)$}\

This section establish the global $H^{s}$ estimates of $(u,\omega,\theta)$ with $s>2$.
\begin{Pros}\label{Pros3.1}
Let $(u_{0},\omega_{0},\theta_{0})$ satisfy the assumptions stated in Theorem {\ref{thm1.1}}. Then the corresponding solution $(u,\omega,\theta)$ of ({\ref{1.3}}) is bounded in $H^{s}(\mathbb{R}^{2})$.
\end{Pros}
\begin{proof}
First, applying the operator $\Lambda^{s}$ to the system ({\ref{1.2}}), and dotting the first equation of ({\ref{1.2}}) by $\Lambda^{s}u$, the second equation by $\Lambda^{s}\omega$, the third equation by $\Lambda^{s}\theta$, respectively, integrating over $\mathbb{R}^{2}$, we have
\begin{align}\label{3.23}
\nonumber &\frac{1}{2}\frac{\mathrm{d}}{\mathrm{d}t}(\|\Lambda^{s}u\|^{2}_{L^{2}}+\|\Lambda^{s}\omega\|^{2}_{L^{2}}+\|\Lambda^{s}\theta\|^{2}_{L^{2}})
+\|\Lambda^{s+1}\omega\|^{2}_{L^{2}}+2\|\Lambda^{s}\omega\|^{2}_{L^{2}}+\|\Lambda^{s+\frac{1}{2}}\theta\|^{2}_{L^{2}}\\
\nonumber &\leq-\int_{\mathbb{R}^{2}}[\Lambda^{s},u\cdot\nabla]u\cdot\Lambda^{s}u\mathrm{d}x
+\int_{\mathbb{R}^{2}}\Lambda^{s}(\nabla\times\omega)\cdot\Lambda^{s}u\mathrm{d}x+\int_{\mathbb{R}^{2}}\Lambda^{s}(e_{2}\theta)\cdot\Lambda^{s}u\mathrm{d}x\\
\nonumber &~~~-\int_{\mathbb{R}^{2}}[\Lambda^{s},u\cdot\nabla]\omega\cdot\Lambda^{s}\omega \mathrm{d}x+\int_{\mathbb{R}^{2}}\Lambda^{s}(\nabla\times u)\cdot\Lambda^{s}\omega \mathrm{d}x\\
\nonumber &~~~-\int_{\mathbb{R}^{2}}[\Lambda^{s},u\cdot\nabla]\theta\cdot\Lambda^{s}\theta \mathrm{d}x+\int_{\mathbb{R}^{2}}\Lambda^{s}(u\cdot e_{2})\cdot\Lambda^{s}\theta \mathrm{d}x\\
 &~~~=J_{1}+J_{2}+J_{3}+J_{4}+J_{5}+J_{6}+J_{7}.
\end{align}
We can estimate $J_{1}-J_{7}$ as follows. By the H\"{o}lder inequality and the commutator estimate in Lemma {\ref{lem2.2}}, we have
\begin{align*}
|J_{1}|\leq C\|[\Lambda^{s},u\cdot\nabla]u\|_{L^{2}}\|\Lambda^{s}u\|_{L^{2}}
\leq C\|\nabla u\|_{L^{\infty}}\|\Lambda^{s}u\|^{2}_{L^{2}}.
\end{align*}
\begin{align*}
|J_{2}+J_{5}|\leq C\|\Lambda^{s+1}\omega\|_{L^{2}}\|\Lambda^{s}u\|_{L^{2}}\leq
\frac{1}{2}\|\Lambda^{s+1}\omega\|^{2}_{L^{2}}+C\|\Lambda^{s}u\|^{2}_{L^{2}},
\end{align*}
\begin{align*}
|J_{3}+J_{7}|\leq C(\|\Lambda^{s}\theta\|^{2}_{L^{2}}+\|\Lambda^{s}u\|^{2}_{L^{2}}).
\end{align*}
By an similar argument to derive $J_{1}$, we can deduce that
\begin{align*}
|J_{4}|&\leq C(\|\nabla u\|_{L^{\infty}}\|\Lambda^{s}\omega\|_{L^{2}}+
\|\Lambda^{s}u\|_{L^{2}}\|\nabla\omega\|_{L^{\infty}})\|\Lambda^{s}\omega\|_{L^{2}}\\
&\leq C(\|\nabla u\|_{L^{\infty}}+\|\nabla \omega\|_{L^{\infty}})(\|\Lambda^{s}u\|_{L^{2}}^{2}+\|\Lambda^{s}\omega\|_{L^{2}}^{2}),
\end{align*}
and
\begin{align*}
|J_{6}|&\leq C\|[\Lambda^{s},u\cdot\nabla]\theta\|_{L^{\frac{4}{3}}}\|\Lambda^{s}\theta\|_{L^{4}}\\
&\leq C(\|\nabla u\|_{L^{4}}\|\Lambda^{s}\theta\|_{L^{2}}+
\|\Lambda^{s}u\|_{L^{2}}\|\nabla\theta\|_{L^{4}})\|\Lambda^{s+\frac{1}{2}}\theta\|_{L^{2}}\\
&\leq C(\|\Lambda^{\frac{3}{2}}\theta\|_{L^{2}}^{2}+\|\Omega\|^2_{L^4})(\|\Lambda^{s}u\|_{L^{2}}^{2}+\|\Lambda^{s}\theta\|_{L^{2}}^{2})
+\frac{1}{2}\|\Lambda^{s+\frac{1}{2}}\theta\|_{L^{2}}^{2}.
\end{align*}
Inserting the above estimates into ({\ref{3.23}}), and applying the logarithmic Sobolev inequality in Lemma \ref{lem2.6}, we have
\begin{align*}
&\frac{\mathrm{d}}{\mathrm{d}t}(\|\Lambda^{s}u\|^{2}_{L^{2}}+\|\Lambda^{s}\omega\|^{2}_{L^{2}}+\|\Lambda^{s}\theta\|^{2}_{L^{2}})
+\|\Lambda^{s+1}\omega\|^{2}_{L^{2}}+\|\Lambda^{s+\frac{1}{2}}\theta\|^{2}_{L^{2}}\\
&\leq C(\|\nabla u\|_{L^{\infty}}+\|\nabla\omega\|_{L^{\infty}}+\|\Lambda^{\frac{3}{2}}\theta\|_{L^{2}}^{2})
(\|\Lambda^{s}u\|^{2}_{L^{2}}+\|\Lambda^{s}\omega\|^{2}_{L^{2}}+\|\Lambda^{s}\theta\|^{2}_{L^{2}})\\
&\leq C(1+\|\Omega\|_{L^{\infty}}\mathrm{log}(e+\|\Lambda^{s}u\|_{L^{2}})+
\|\nabla\omega\|_{L^{\infty}}+\|\Lambda^{\frac{3}{2}}\theta\|_{L^{2}}^{2})\\
&~~~\times(\|\Lambda^{s}u\|^{2}_{L^{2}}+\|\Lambda^{s}\omega\|^{2}_{L^{2}}+\|\Lambda^{s}\theta\|^{2}_{L^{2}}).\\
\end{align*}
By the logarithmic Gronwall's inequality in Lemma \ref{lem2.7}, one has
\begin{align*}
&(\|u\|^{2}_{H^{s}}+\|\omega\|^{2}_{H^{s}}+\|\theta\|^{2}_{H^{s}})+\int^{T}_{0}\|\nabla\omega\|^{2}_{H^{s}}\mathrm{d}\tau+\int^{T}_{0}\|\Lambda^{\frac{1}{2}} \theta\|^{2}_{H^{s}}\mathrm{d}\tau\\
&\leq C(\|u_{0}\|_{H^{s}},\|\omega_{0}\|_{H^{s}},\|\theta_{0}\|_{H^{s}},T).
\end{align*}
The proof of Proposition {\ref{Pros3.1}} is completed. With the Proposition {\ref{Pros3.1}} at our disposal, by a standard continuation argument of local solutions, we complete the proof of Theorem \ref{thm1.1}.

\end{proof}

\vskip .3in
\textbf{Declarations}

\textbf{Conflict of interest}: The authors declared that they have no conflict of interest.

\vskip .2in
\textbf{Acknowledgements.}
The authors would like to thank Professor Zhuan Ye for discussion on this topic.

\end{document}